\documentclass[11pt,final]{amsart}
\usepackage[utf8]{inputenc}
\usepackage[T1]{fontenc}
\usepackage[letterpaper,centering]{geometry}
\usepackage{lmodern}
\usepackage{eucal}
\usepackage{mathrsfs}
\usepackage{enumerate}
\usepackage{amsmath,amssymb,amsfonts}
\usepackage{xcolor}
\definecolor{darkblue}{rgb}{0,0,0.3}
\definecolor{darkgreen}{rgb}{0,0.4,0}
\usepackage[colorlinks=true,linkcolor=darkblue, citecolor=darkgreen, unicode,pdfborder={0 0 0}]{hyperref}

\usepackage[ps,arrow,matrix,tips,line, curve]{xy}
{\setbox0\hbox{$ $}}\fontdimen16\textfont2=\fontdimen17\textfont2
\entrymodifiers={+!!<0pt,\the\fontdimen22\textfont2>}
\SelectTips{cm}{11}
\usepackage{enumitem}

\setlist[enumerate]{label={\upshape(\arabic*)},topsep=.7ex, leftmargin=*}
\setlist[itemize]{leftmargin=*}

\usepackage{ifdraft}
\ifdraft{
	\usepackage[notcite,notref]{showkeys}
	\addtolength{\hoffset}{1.5cm}
	
}{}

\usepackage{etoolbox}

\theoremstyle{plain}

\newtheorem{thm}{Theorem}[section]

\newtheorem{lemma}[thm]{Lemma}
\newtheorem{cor}[thm]{Corollary}
\newtheorem{prop}[thm]{Proposition}

\theoremstyle{definition}
\newtheorem{question}[thm]{Question}

\newtheorem{remark}[thm]{Remark}

\numberwithin{equation}{section}

\input cyracc.def
\DeclareFontFamily{U}{russian}{}
\DeclareFontShape{U}{russian}{m}{n}
{ <5><6> wncyr5
	<7><8><9> wncyr7
	<10><10.95><12><14.4><17.28><20.74><24.88> wncyr10 }{}
\DeclareSymbolFont{Russian}{U}{russian}{m}{n}
\DeclareSymbolFontAlphabet{\mathcyr}{Russian}
\makeatletter
\let\@math@cyr\mathcyr
\renewcommand{\mathcyr}[1]{\@math@cyr{\cyracc #1}}
\makeatother

\newcommand{\rank}{{\mathrm{rank}}}

\makeatletter
\def\myrightarrow{{\setbox\z@\hbox{$\rightarrow$}\dimen0\ht\z@\multiply\dimen0 6\divide\dimen0 10\ht\z@\dimen0\box\z@}}
\def\myrightarrowfill@{\arrowfill@\relbar\relbar\myrightarrow}
\def\myleftarrow{{\setbox\z@\hbox{$\leftarrow$}\dimen0\ht\z@\multiply\dimen0 6\divide\dimen0 10\ht\z@\dimen0\box\z@}}
\def\myleftarrowfill@{\arrowfill@\myleftarrow\relbar\relbar}
\newcommand{\myxrightarrow}[2][]{\ext@arrow 0359\myrightarrowfill@{#1}{#2}}
\newcommand{\myxleftarrow}[2][]{\ext@arrow 3095\myleftarrowfill@{#1}{#2}}
\makeatother
\newcommand{\mtilde}{{\mathchoice
		{\widetilde{m}}
		{\widetilde{m}}
		{\rlap{$\scriptscriptstyle{m}$}\vphantom{\raise0pt\hbox{$m$}}\smash{\lower2.5pt\hbox{$\scriptscriptstyle\widetilde{\phantom{\scriptscriptstyle{m}}}$}}}
		{\rlap{$\scriptscriptstyle{m}$}\vphantom{\raise.2pt\hbox{$m$}}\smash{\lower2.05pt\hbox{$\scriptscriptstyle\widetilde{\phantom{\scriptscriptstyle{m}}}$}}}}}

\newcommand{\Mtilde}{{\mathchoice
		{\rlap{$M$}\mkern1mu\smash[b]{\lower.5pt\hbox{$\widetilde{\phantom{M}}$}}\mkern-1mu}
		{\rlap{$M$}\mkern1mu\smash[b]{\lower.5pt\hbox{$\widetilde{\phantom{M}}$}}\mkern-1mu}
		{\rlap{$\scriptstyle{M}$}\mkern1mu\smash[b]{\lower.5pt\hbox{$\widetilde{\phantom{\scriptstyle{M}}}$}}\mkern-1mu}
		{\widetilde{M}}}}
\newcommand{\etabar}{{\bar\eta}}

\newcommand{\et}{{\text{ét}}}

\newcommand{\A}{{\mathbf A}}

\ifdef{\C}%
{{\renewcommand{\A}{{\mathbb{A} }}}
\renewcommand{\C}{{\mathbf C}}}%
{\newcommand{\C}{{\mathbf C}}}

\newcommand{\cF}{\mathrm F}

\newcommand{\cFplus}{\mathrm F_+}
\newcommand{\cFconst}{\mathrm F_\const}

\newcommand{\const}{\mathrm{const}}

\newcommand{\Pic}{\mathrm{Pic}}
\newcommand{\Br}{\mathrm{Br}}

\renewcommand{\phi}{\varphi}
\renewcommand{\emptyset}{\varnothing}

\newcommand{\Coker}{{\mathrm{Coker}}}
\newcommand{\Hom}{{\mathrm{Hom}}}

\newcommand{\chapeau}{{\rlap{\smash{\hbox{\lower4pt\hbox{\hskip1pt$\widehat{\phantom{u}}$}}}}}}
\newcommand{\Picplushat}{\Pic_+^{{\smash{\hbox{\lower4pt\hbox{\hskip0.4pt$\widehat{\phantom{u}}$}}}}}}
\newcommand{\PicplusAhat}{\Pic_{+,\A}^{{\smash{\hbox{\lower4pt\hbox{\hskip.4pt$\widehat{\phantom{u}}$}}}}}}

\newcommand{\Pichat}{\Pic^{{\smash{\hbox{\lower4pt\hbox{\hskip0.4pt$\widehat{\phantom{u}}$}}}}}}

\DeclareMathOperator{\Ker}{Ker}
\DeclareMathOperator{\inv}{inv}

\ifdef{\G}%
{\renewcommand{\G}{{\mathcal{G}}}}%
{\newcommand{\G}{{\mathcal{G}}}}

\makeatletter
\renewcommand{\tocsection}[3]{%
	\indentlabel{\@ifnotempty{#2}{\bfseries\ignorespaces#1 #2\quad}}\bfseries#3}

\renewcommand{\tocsubsection}[3]{%
	\indentlabel{\@ifnotempty{#2}{\hspace{1.6em}\ignorespaces#1 #2\quad}}#3}
\makeatother

\makeatletter\let\@wraptoccontribs\wraptoccontribs\makeatother

\usepackage{threeparttable}
\usepackage{threeparttable, tablefootnote}

\usepackage{silence}
\newcommand\kbar{{\overline{k}}}

\newcommand\ZZ{\mathbb{Z}}
\newcommand\QQ{\mathbb{Q}}

\newcommand\GG{\mathbb{G}}

\newcommand\sm{\mathrm{sm}}

\DeclareMathOperator\ord{ord}

\newtheorem*{ack}{Acknowledgements}

\newtheorem*{terminology}{Terminology}

\usepackage{soul}
\begin{document}

\title[strong approximation for the intersection of  two quadrics]
{strong approximation for the intersection of  two quadrics}

\author{Dasheng Wei}

\address{Academy of Mathematics and System Science, CAS, Beijing 100190,
  P.\ R.\ China \emph{and} School of mathematical Sciences, University of  CAS, Beijing
  100049, P.\ R.\ China}

\email{dshwei@amss.ac.cn}

\author{Jie Xu}

\address{Academy of Mathematics and System Science, CAS, Beijing 100190,
	P.\ R.\ China}

\email{xujie2020@amss.ac.cn}

\author{Yi Zhu}

\address{Bethesda, MD, USA 20817}

\email{math.zhu@gmail.com}

\date{December 12, 2024}

\begin{abstract}
	

 We study strong approximation for the intersection of two affine quadrics. As its application, we prove the fibration method for weak approximation over number fields of rank four with nonsplit fibers split by quadratic extensions.
\end{abstract}


\maketitle

\section{Introduction}




The paper is mainly devoted to the study of the following  fibration method for weak approximation over number fields.

\begin{question}\label{question:fibration}
Let $X$ be a smooth proper variety over a number field $k$
and let $f: X \to \mathbb{P}^1_k$ be a dominant morphism whose geometric generic fiber is rationally connected.
Assume that $X_c(k)$ is dense in $X_c(\A_k)^{\Br(X_c)}$ for all but finitely many $c \in \mathbb{P}^1(k)$, where
$X_c = f^{-1}(c)$. Does it follow that $X(k)$ is dense in $X(\A_k)^{\Br(X)}$?
\end{question}

Question \ref{question:fibration} has been extensively studied. One main approach
consists in applying the theory of descent developed by
Colliot-Th\'el\`ene and Sansuc~\cite{CTS87} to reduce the problem to certain
torsors associated with the vertical Brauer group of $X$ relative to~$\mathbb P^1_k$.
This approach has been applied successfully in many
cases, including Ch\^atelet surfaces~\cite{CTSSD1}, some conic and
quadric bundles~\cite{bms}, and various toric
bundles \cite{heathbrownskorobogatov,CTSa,derenthalsmeetswei, browningmatthiesen}.

Another approach, known as the fibration method, was initiated by
Harari~\cite{Har94} and has been further developed by
Wittenberg and Harpaz~\cite{HW}. The fibration method has been successfully applied to various families of rationally connected varieties, including conic bundles~\cite{CT90,HW} and some quadric bundles~\cite{BS19,HWW22}.

The known cases for fibration method are mostly for fibrations of low ranks. Recall here that the \emph{rank} of a fibration $f:X\to \mathbb{P}_k^1$ is defined to be 
the sum of degrees of  closed points of $\mathbb{P}_k^1$ above which the fiber of $f$ is not split.
The definition of split fibers can be found in work of  Skorobogatov \cite{Sko96}, where the notion was originally introduced to the subject.

\begin{thm}[\cite{Har94,HW,BS19,HWW22}]
	Question~\ref{question:fibration} has a positive answer for each of the following cases :
	\begin{enumerate}[label={\upshape(\roman*)}]
		\item the rank of $f$ is at most $2$;
		\item the rank of $f$ is $3$ and every fiber~$X_m$ is split by a quadratic extension of~$k(m)$;
		\item the rank of $f$ is $3$, one fiber~$X_m$ lies above a rational point of~$\mathbb P^1_k$
		and every remaining fiber~$X_m$  is split by a quadratic extension of~$k(m)$.
	\end{enumerate}	
\end{thm}

Recently, Harpaz, Wittenberg and the first author \cite{HWW22} proved that Question~\ref{question:fibration} has a positive answer when all non-split fibers are split by cyclic extensions of the base fields of the fibers under Schinzel's hypothesis. For unconditional results, the strongest result is due to Browning and Matthiesen \cite{browningmatthiesen}. 

\begin{thm}[\cite{HW,browningmatthiesen}]
	\label{thm:bm}
	Question~\ref{question:fibration} has a positive answer when the base field is $\mathbb{Q}$ and all non-split fibers lie over rational points.
\end{thm}

In this paper, we investigate the rank four case of Question~\ref{question:fibration}. This can be viewed as a generalization of Colliot-Th\'el\`ene, Sansuc and Swinnerton-Dyer's result on Ch\^atelet surfaces \cite{CTSSD1} and Colliot-Th\'el\`ene's result on conic bundles of rank four \cite{CT90}.

\begin{thm}\label{cor:main-rank}
	Question~\ref{question:fibration} has a positive answer when
	$\rank(f)=4$ and  all fibers split by a quadratic extension.
\end{thm}

By Theorem ~\ref{cor:main-rank}, we have the following example for quadric surface bundles with at most four nonsplit fibers split by quadratic extensions. 

\begin{cor}\label{exa:sko}
	Let $X/\mathbb P^1_k$ be the projective quadric surface bundle given by $$x^2-a(t)y^2+P(t)(z^2-b(t)w^2)=0$$
	where $a(t),b(t)$ and $P(t)$ are nonzero square-free polynomials in $k[t]$ of even degree. Let $\gcd(a(t), b(t)P(t))=1 .$ Suppose that the degree of $P(t)$ is at most $4$, then $X(k)$ is dense in $X(\A_k)^{\Br(X)}$.
\end{cor} 

To prove Theorem ~\ref{cor:main-rank}, we follow the fibration method developed by Harpaz, Wittenberg and the first author of the paper \cite[Cor. 4.7]{HWW22}. The rank four case reduces to study strong approximation for the intersection of two affine quadrics, which is the main effort of this paper. In fact, its proof immediately follows case (iii) in Theorem \ref{thm:main'}.

\begin{thm}\label{thm:main'}
	Let $k$ be a number field and $v_0$ a place of $k$.  Let $V\subset \mathbb P^n_k, n\geq 5$ be the pure geometrically integral intersection of two quadratics which is not a cone. Let $\widetilde V$ be the affine cone of $V$ and $Z$ a closed subset of $\widetilde V$ of codimension at least $2$. Let $\widetilde V^{\sm}$ be the smooth locus of $\widetilde V$ and  $\widetilde U=\widetilde V^{\sm}\setminus Z$. 
	
	Suppose that $V$  has only finitely many singular points. Then: 
	\begin{enumerate}[label=(\roman*)]
		\item Suppose $V^\sm(k)\neq \emptyset$ and $n=5$, then $\widetilde U$ satisfies strong approximation with Brauer-Manin obstruction off $v_0$.
		\item Suppose $V^\sm(k)\neq \emptyset$ and $n= 6$, then $\widetilde U$ satisfies strong approximation off $v_0$.
		\item  Suppose $n\geq 7$, then $\widetilde U$ satisfies strong approximation off $v_0$.
	\end{enumerate}
\end{thm}

In particular, we have the following corollary for nonsingular intersection of two affine quadrics.
\begin{cor}
 Let $v_0$ be a place of $k$.  Let $\widetilde V\subset \mathbb A^{n+1}_k\setminus{(0,\cdots,0)}$ be the  variety defined by $$Q_1(x_0,\cdots,x_n)=Q_2(x_0,\cdots,x_n)=0,$$ where $Q_1$ and $Q_2$ are quadratic forms. Let  $Z$ be a closed subset of $\widetilde V$ of codimension at least $2$. If $n\geq 7$ and $\widetilde V$ is nonsingular, then $\widetilde V\setminus Z$ satisfies strong approximation off $v_0$.
\end{cor}


\begin{terminology}
	Notation and terminology are standard. Let $k$ be a number field, $\Omega_k$ the set of all places of $k$ and  $\infty$ the set of all archimedean places of $k$. 
	Let $\mathcal O_k$ be the ring of integers of $k$ and $\mathcal O_S$ the $S$-integers of $k$ for a finite set $S$ of $\Omega_k$ containing $\infty$. For each $v\in \Omega_k$, the completion of $k$ at $v$ is denoted by $k_v$, the completion of $\mathcal O_k$ at $v$ by $\mathcal O_v$ and the residue field at $v$ by $k(v)$ for $v\not\in \infty$. Let ${\mathbf A}_k$ be the adele ring of $k$.
	
	  Let $X$ be a smooth variety over $k$. Denote  $X(\mathbf A_k)^B$ to be the set of all $(x_v)_v \in X(\mathbf A_k)$ satisfying $\sum_{v \in \Omega_k} \inv_v(A(x_v))
	= 0$ for each $A$ in the subgroup $B$ of the Brauer group $\Br(X) =
	H^2_\text{\'et}(X,\mathbf G_{m})$ of $X$, where the map $\inv_v : \Br(k_v)
	\to \QQ/\ZZ$ is the invariant map from local class field theory.
	The subgroup $\Br_0(X)$ of constant elements in the Brauer group is
	the image of the natural map $\Br(k) \to \Br(X)$. The algebraic
	Brauer group $\Br_1(X)$ is the kernel of the natural map $\Br(X) \to
	\Br(X_{\kbar})$, where $X_{\kbar} = X \times_k \kbar$.


	Let $S$ be a finite set of places of $k$. We say that
	\emph{strong approximation} holds for $X$ off $S$ if
	the image of the set $X(k)$ of rational points on $X$ is dense in the space
	$X(\mathbb A_k^S)$ of adelic points on $X$ outside $S$. Strong approximation for $X$
	off $S$ implies the Hasse principle for $S$-integral points on any
	$S$-integral model of $X$.  We say that
	$X(\A_k^S)$, \emph{strong approximation with Brauer-Manin obstruction} holds for $X$ off $S$ if
		$X(k)$ is dense in the image of $X(\mathbf A_k)^{\Br}$ in 
		$X(\A_k^S)$. Here, $X(\mathbf A_k)^{\Br}$ denotes the set of adelic points orthogonal to the Brauer group with respect to all finite étale covers of $X$; see \cite[Definition~2.1]{Skobogatov2001} for a precise definition. 

	For  a geometrically integral variety $X$ over a number field $k$, we denote  $X^\sm$ to be  its smooth locus.
	Let $f: \widetilde {X}\to X$ be a resolution of singularities.   If $ X^{\sm}(k)$   is dense 	the image of $f(\widetilde {X}(\mathbf A_k)^{\Br(\widetilde {X})})$
	(resp. $f(\widetilde {X}(\mathbf A_k)^{\Br_1(\widetilde {X})})$) in $X(\mathbf A_k^S)$, we say that
	\emph{central strong approximation} with (resp.  algebraic) Brauer-Manin obstruction off $S$ holds for X. The definition does not depend on the choice of the resolution of singularities by \cite[Proposition 2.3]{SX}. 

\end{terminology}

\begin{ack}
	The first author is supported by National Key R\&D Program of China (Grant No. 2020YFA0712600) and National Natural Science Foundation of China (Grant Nos. 12371014 and 12231009). 
	The third author wishes to express deep gratitude to Professor Shiu-Chun Wong, who sadly passed away recently, for his unwavering passion for teaching mathematics, his mentorship and lasting influence on the third author's journey in mathematics.  
  \end{ack}

\section{Strong Approximation for singular intersection of two quadrics}
In this section, we collect some results which will be used in the proof of Theorem ~\ref{thm:main'}. The first lemma is a generalization of the easy fibration method by \cite[Prop. 3.1]{CTX13}.
\begin{lemma}\label{lem:SA-overAn} 
	Let $S$ be a  finite non-empty set of places of $k$.	Let $L/k$ be a finite \'etale algebra.
	Let $f:Y\to  R_{L/k}(\mathbb A^1_L)\setminus F$ be a surjective morphism, where $F \subset R_{L/k}(\mathbb A^1_L)$ is a closed subset of codimension $2$. Let $U:=R_{L/k}(\mathbb G_{\text{m},L})$.  Let $W\subset U\setminus F$ be an open subset such that the fiber of $f_W: f^{-1}(W)\to W$ is smooth.  We suppose:
	\begin{itemize}
		\item all geometric fibers of $f$ above $U\setminus F$ are split;
		\item there exists a finite field extension $K$ of $k$ such that for any closed point $m$ in $R_{L/k}(\mathbb A^1_L)\setminus (U\cup F)$, the fiber $f_m\otimes_k K$ is split.
		\item the fiber of $f$ above any $k$-point of $W$ satisfies strong approximation off $S$;
		\item for any $v\in S$, $f^{-1}(W)(k_v)\to W(k_v)$ is onto.
	\end{itemize}  
	
	Then  $Y$ satisfies strong approximation with algebraic Brauer--Manin obstruction off $S$.
\end{lemma}
\begin{proof} Let $K'=L\otimes_k K$, let $T$ be the torus $R_{K'/k} (\mathbb G_{\text{m},K'})$ and $T'$  the norm one torus $R_{K'/k}^1 (\mathbb G_{\text{m},K'})$.	Let $U'=U\setminus F$.
	The projection $$U'\times_{R_{L/k} (\mathbb G_{m,L})} T \to U'$$ gives a $T'$-torsor over $U'$, where  the map  $T\to R_{L/k} (\mathbb G_{m,L})$ is given by sending $a\in T$ to $ \alpha N_{K'/L}(a)$, where $\alpha \in L^*$. First we fix $\alpha=1$. We denote by $\tau \in H^1_{\et}(U',T')$ its isomorphism class, the cup products
	$\tau\cup \chi$ give an element in $\Br_1(U')$ for any  $\chi \in H^1(k,\widehat T')$.
	Obviously $H^1(k,\widehat T')$ is finite. 
	Let~$A$ be the finite subgroup of $\Br_1(f^{-1}(U'))$ generated by $f^*(\tau\cup \chi)$
	where  $f^*:\Br(U')\to \Br(f^{-1}(U'))$ and $\chi \in H^1(k,\widehat T')$.
	For any $(y_v)_v \in Y(\A_k)^{A\cap\Br_1(X)}$,
	there exists $(y'_v)_v \in Y(\A_k)^{A}$
	arbitrarily close to $(y_v)_v$ in~$Y(\A_k)$ by Harari's formal lemma (\cite[Th\'eor\`eme~1.4]{CT03}).
	Since $A$ is finite, we may assume that all $y'_v$  belong to~$f^{-1}(W)$ by  the implicit function theorem.

	Let us apply open descent theory
	to the projection $f^{-1}(U')\times_{R_{L/k} (\mathbb G_{m,L})} T \to f^{-1}(U')$. 
	According to \cite[Theorem~8.4, Proposition~8.12]{HS13},
	the adelic point
	$(y'_v)_v$ can be lifted to an adelic point $(z_v)_v$ of some
	twist of this torsor.  That is to say, there exist $\alpha \in L^*$ and $(a_v)_v \in \A_{K'}$ such that
	$\alpha N_{K'/L}(a_v)=f(y'_v)$ for any place  $v$.

	Let $F'$ be the inverse image of the union of  $F$ and the singular locus of $R_{L/k}(\mathbb A^1_L)\setminus U$ in $R_{K'/k} (\mathbb A^1_{ K'})$. Then $F'$ has codimension $2$ in $R_{K'/k} (\mathbb A^1_{ K'})$. 	
	Since $R_{K'/k} (\mathbb A^1_{ K'})\setminus F'$ satisfies strong approximation off $v_0$ by \cite[Lemma ~2.1]{weitorus}, we may choose an element $a \in {K'}^*$ which corresponds a $k$-point in $R_{K'/k} (\mathbb G_{\text m,K'})$, such that $a$ is very close to $(a_v)_v$ in  $\A^S_{K'}$.
	Let $\mu_0:=N_{ K'/L}(a)$, by shrinking the element $a$, we may assume that~$\mu_0\in W$.
	
	Let~$S'\supset S$ be a finite set of places of~$k$ containing  the archimedean places,  all finite places where $K'/k$ is ramified and all places at which we want to approximate $(y_v)_{v}$.	Enlarging $S'$, we may assume that all maps can be extended to their $\mathcal O_{S'}$-models.	 
	Suppose $v\not \in S'$ and $\ord_v(N_{L/k}(\mu_0))> 0$. Since $a \not \in \mathscr F'$, one has $\mu_0\not \in \mathscr F_L$, where $\mathscr F'$ and  $\mathscr F_L$ are respectively the standard $\mathcal O_{S'}$-model of $F'$ and of the singular locus of $R_{L/k}(\mathbb A^1_L)\setminus U$. Therefore, $L$ has a place $w$ of degree $1$ above $v$ such that $\ord_w(\mu_0)> 0$ and $K$ also has a place of degree $1$ above $v$,  hence  $Y_{\mu_0}$ is split at $v$ by the second assumption. If $v\not \in S'$ and $\ord_v(N_{L/k}(\mu_0))= 0$, $Y_{\mu_0}$ is split at $v$ since $\mu_0 \in \mathscr U$, where $\mathscr U$ is $\mathcal O_{S'}$-model of $U$. Therefore, by the Lang--Weil--Nisnevich estimate~\cite{langweil, nisnevic},
	we may assume that $Y_{\mu_0}$ has local integral points at any $v\not \in S'$. If $v\in S'\setminus S$, we may choose a local point on $ Y_{\mu_0}(k_v)$ which is very close to~$y'_v$ since $Y_{\mu_0}$ is smooth. If $v\in S$, we may choose a local point $y'_v\in Y_{\mu_0}(k_v)$ by the surjectivity of $f^{-1}(W)(k_v)\to W(k_v)$. The proof then follows from that $Y_{\mu_0}$ satisfies strong approximation off $S$.
\end{proof}

For any projective variety $X\subset \mathbb P^n_k$ over a field $k$, let $\widetilde X$ be the punctured affine cone of $X$, which can be viewed as a $\GG_{\text{m}}$-torsor over $X$ of type $\chi: \mathbb Z \rightarrow \Pic(X), 1\mapsto -H$, where $H$ is the restriction on $X$ of the hyperplane divisor of $\mathbb P^n_k$. 

For simplicity, in what follows, we will denote the type of any $\GG_m$-torsor by the corresponding element of the Picard group (e.g., $1\mapsto a$ for $a\in \Pic(X)$).

\begin{lemma}\label{lem:SA-U} 
	Let $M=\{p_1,\cdots, p_s\}$ be the set of $s$ distinct $\kbar$-points in $\mathbb A^2_k$ with $s\geq 1$ ($M$ is defined over $k$). Let $U:=Bl_M(\mathbb A^2_k)$ be the blow-up of $\mathbb A^2_k$ along $M$. Let $l_i$ be the exceptional curve above $p_i$. Let $\widetilde U$ be the $\GG_{\text{m}}$-torsor over $U$ of type $1\mapsto -(l_1+\cdots +l_s)$. 	
	
	Then,	
	\begin{enumerate}
		\item[(i)] $\widetilde U$ is isomorphic to the variety of $\mathbb A^4_k$ defined by the equation 
		\begin{equation}\label{eq:Bl4}
			f(x)\lambda=y\mu,
		\end{equation} where $(x,y) \in \mathbb A^2_k$, $(\lambda,\mu)\in \mathbb A^2_k\setminus {(0,0)}$ and $f(x)$ is a separable polynomial of degree $s$. 
		\item[(ii)]  $\Br(\widetilde U_{\kbar})=0$ and $\Br(\widetilde U)/\Br_0(\widetilde U)$ is finite.
		
		\item[(iii)] 	$\widetilde U\setminus Z$ satisfies strong approximation with algebraic Brauer-Manin obstruction off $v_0$, where $Z\subset \widetilde U$ has codimension  at least 2.
	\end{enumerate}
\end{lemma}
\begin{proof} Choose a linear change of coordinates, we may assume that the $x$-coordinates of $p_1,\cdots, p_s$ are different, then $x$-coordinates of $p_1,\cdots, p_s$ are just all roots of a separable polynomial $f(x)$ of degree $s$. Then $y$-coordinates of $p_i$ satisfies $y=g(x)$, here $x$ is the $x$-coordinates of $p_i$ and $g(x)$ is a $k$-polynomial of degree at most $s-1$. By the change of coordinates $x'=x, y'=y-g(x)$. Therefore we may assume that $M$ is the set defined by $f(x)=0, y=0$. Hence $U$ is isomorphic to the variety of $\mathbb A^2_k \times \mathbb P^1_k$ defined by the equation $$f(x)\lambda=y\mu,$$ where $(x,y) \in \mathbb A^2_k$, $(\lambda:\mu)\in \mathbb P^1$. Let  $\widetilde U$ be of the form described as in (1), then $\widetilde U$ is the $\GG_{\text{m}}$-torsor over $U$ of type $1\rightarrow -(l_1+\cdots +l_s)$.
	
	We have a fibration $f:\widetilde U \to \mathbb A^1_k$ by sending $(x,y,\lambda,\mu)$ to $\mu$.  
	The generic fiber $\widetilde U_\eta$ is isomorphic to an affine space over the function field $k(\eta)$, so over $\kbar(\eta)$ it is an affine space over an algebraically closed field. By Tsen's theorem (see, e.g., \cite[III.2.8]{Serre2002}), the Brauer group of such a function field vanishes, i.e., $\Br(\widetilde U_{\eta, \kbar}) = \Br(\kbar(\eta)) = 0$, and thus by Hochschild-Serre's spectral sequence, $\Br(\widetilde U_\kbar) = 0$.
	Note that $\Pic (\widetilde U_\kbar)\cong \Coker[\mathbb Z\to \bigoplus_{i=1}^s\mathbb Z l_i]$, where the map sends $1$ to $(1,1,\ldots,1)$, as follows from the standard computation of the Picard group of a blow-up at $s$ points (see, e.g., \cite[II.6.5, Example 6.7.2]{Hartshorne}), which implies $\Pic (\widetilde U)$ is free and finitely generated, hence  
	$\Br(\widetilde U)/\Br_0(\widetilde U)$ is finite. 
	
	Let $V\subset \mathbb A^1_k$ be an open subset defined by $\mu \neq 0$. Note that $f^{-1}(V)$ is a product of $V$ and $\mathbb A^1_k$. 
	To apply Lemma~\ref{lem:SA-overAn}, observe that the morphism $f: \widetilde U \to \mathbb{A}^1_k$ is surjective, all geometric fibers above $V$ are isomorphic to $\mathbb{A}^1_k$ (hence split), and for any $k$-point of $V$, the fiber satisfies strong approximation off $v_0$ (since it is an affine line). 
	Moreover, for any place $v$, the map $\widetilde U(k_v) \to V(k_v)$ is surjective by the structure of the equation. 
	Thus, all the hypotheses of Lemma~\ref{lem:SA-overAn} are satisfied, and the proof of case $(3)$ follows.
\end{proof}

Next, we collect some results on strong approximation for $\GG_{\text{m}}$-torsors over singular del Pezzo surfaces of degree $4\leq d \leq 8$. Singular del Pezzo surfaces of degree $d$ are classified in \cite{CorT88}, where the singularities are all du Val singularities. The following lemma is a direct consequence of Lemma~\ref{lem:SA-U}.

\begin{lemma}\label{lem: S-bl}
	Suppose  $V=Bl_{M_s}(S)$, where $S$ is $\mathbb P^2$, a quadric surface, or $F_2$ (where $F_2$ denotes the Hirzebruch surface $\mathbb F_2$), and $M_s$ is a $k$-subset of $S(\kbar)$ consisting of $s$ distinct points, $s\geq 1$.
	Let $\widetilde V$ be the $\GG_{\text{m}}$-torsor over $V$ of type $K_V$, where $K_V$ is the canonical  divisor. Let $Z$ be a closed subset of $\widetilde V$ of codimension at least $2$. 
	
	Then, $\widetilde V \setminus Z$ satisfies strong approximation with algebraic Brauer-Manin obstruction off $v_0$, where $v_0$ is a place of $k$.
\end{lemma}
\begin{proof}
	We may assume $S(k)\neq \emptyset$. Therefore, there is an open subset of $S$ containing $M_s$ which is isomorphic to the affine space of dimension $2$. Then $Bl_{M_s}(\mathbb A^2_k)\subset Bl_{M_s}(S)$ as an open subset. The proof then follows from Lemma~\ref{lem:SA-U} ii)+iii).
\end{proof}

\begin{thm}\label{thm:SA-general-type}
	Let $V$ be a singular del Pezzo surface of degree $4\leq d \leq 8$, and $V$ is not  an Iskovskih surface when $d=4$. Let $\widetilde V$ be the $\GG_{\text{m}}$-torsor over $V$ of type $K_V$, where $K_V$ is the canonical  divisor. Let $\Delta$ be the preimage of the singular locus of $V$ in $\widetilde V$. Let $Z$ be a closed subset of $\widetilde V$ of codimension at least $2$. 
	
	Suppose that $\Delta \cap Z$ is a finite set.  Then, $\widetilde V \setminus Z$ satisfies central strong approximation with algebraic Brauer-Manin obstruction off $v_0$ (see Terminology in Section~1 for the definition), where $v_0$ is a place of $k$.
\end{thm}
\begin{proof} Let $\pi:V'\to V$ be the minimal resolution of $V$. It is a generalized del Pezzo surface. Since $V$ has du Val singularity, $\pi^*K_V=K_{V'}$. Let $\widetilde{V'}$ be the pullback $\GG_m$-torsor $\pi^*(\widetilde{V})$ over $V'$ and it is of the type $K_{V'}$. Let  $p:\widetilde{V'}\to \widetilde{V}$ be the natural morphism. The close subset $p^{-1}(Z)$ is of codimension at least $2$ by assumption. It suffices to prove the strong approximation holds for $\widetilde{V'}\setminus p^{-1}(Z) $.  
	
	By the classification of generalized del Pezzo surfaces of degree $4$ (see \cite[Proposition 6.1]{CorT88}), such a surface is birational to either $\mathbb{P}^2$, a quadric surface, or the Hirzebruch surface $F_2$, except in the case of Iskovskih surfaces. 
	For clarity, in the following list, "case (1)" refers to the specific configuration described as case (1) in \cite[Proposition 6.1]{CorT88}. The classification is as follows:

	\begin{itemize}
		\item the minimal desingularization $V'$ of $V$ is just $Bl_{M_s}(S)$:	
		\subitem $S=\mathbb P^2 (s=5)$: case (2);
		\subitem $S=\text{a quadric surface} (s=4)$: case (1), (4), (6), (9);
		\subitem $S=\mathbb F_2 (s=4)$: case (3) (not an Iskovskih surface), (8);
		\item the minimal desingularization $V'$ of $V$ contains an open subset $Bl_{M_s}(S)\setminus Z'$, where $Z'$ is a finite set:
		\subitem $S=\mathbb P^2$: case (7)($s=2$ or $3$), (12)($s=1$), (15) ($s=1$).
		\subitem $S=\text{a quadric surface}$: case (5)($s=2$), (10)($s=2$), (11)($s=2$ or $3$).
		\subitem $S=\mathbb F_2$: case (13)($s=2$),  (14)($s=2$).
	\end{itemize}
	For the first class, the proof follows from Lemma \ref{lem: S-bl}. For the second class, the complement of $Bl_{M_s}(S)$ is a union of exceptional rational curves and thus we have the restriction of $K_{V'}$ is $K_{Bl_{M_s}(S)}$. So the proof also follows from Lemma \ref{lem: S-bl}. 
	
	If $5\leq d\leq 7$, the proof is similar as above by \cite[Proposition 8.1, 8.3 and 8.5]{CorT88} which give the classification of singular del Pezzo surfaces of degree $5\leq d\leq 7$. If $d=8$ ($i.e.$, the minimal desingularization of $V$ is $F_2$), then $\Pic(\widetilde V)\cong \ZZ/2$. Any universal torsor of $\widetilde V$ is also a universal torsor of $V$. The universal torsor of $V$ is unique and isomorphic to a conic cone, whose minimal desingularization contains an affine space of dimension 3, the proof then
	follows from \cite[Lemma 2.1]{weitorus} 
    and descent theory \cite{CTS87}.
\end{proof}

\begin{remark}
	If $V$ is a smooth or singular del Pezzo surface of degree $4$, then $V\subset \mathbb P^4_k$ is the intersection of two quadrics $\Psi_1(x_0,\cdots,x_4)=\Psi_2(x_1,\cdots,x_4)=0,$ where $\Psi_1, \Psi_2$ are quadratic forms. The $\GG_{\text{m}}$-torsor $\widetilde V$ of type $K_V$  is just the punctured affine cone in $\mathbb A^5_k\setminus \{(0,\cdots,0)\}$ defined by $\Psi_1(x_0,\cdots,x_4)=\Psi_2(x_1,\cdots,x_4)=0$.
\end{remark}

Lastly, we prove strong approximation for certain singular intersection of two quadrics in $\mathbb P^n_k$ with $n\geq 4$.

\begin{prop}\label{prop:genral}
	Let~$k$ be a field of characteristic~$0$ and $n\geq 4$ an integer. Let
	$\Psi_1(x_0,\cdots,x_n)=x_0x_1+\Psi'_1(x_1,\cdots,x_n)$ and $\Psi_2(x_1,\cdots,x_n)$
	are quadratic forms. 
	Suppose  that  $\Psi_2$ has rank~$n$ and that the projective variety $C\subset \mathbb P^{n-2}$ defined by $$\Psi'_1(0,x_2,\cdots,x_n)=\Psi_2(0, x_2,\cdots,x_n)=0$$ is a smooth complete intersection (if $n=4$, then $C$ is a union of $4$ distinct points). 
	Let~$\widetilde V\subset \mathbb A^{n+1}\setminus {(0,\cdots,0)}$ be defined by 
	\begin{equation}\label{eq:general}
		\Psi_1(x_0,\cdots,x_m)=\Psi_2(x_0,\cdots,x_n)=0.
	\end{equation}
	Let $Z\subset \widetilde V$ has codimension at least $2$ and let  $\widetilde U=\widetilde V^{\sm}\setminus Z$. 
	Then	
	\begin{enumerate}[label=(\roman*)]
		\item $\widetilde V$ is geometrically integral with the singular locus $\{(a,0,\cdots,0):a\in \kbar^*\}$;
		\item  We have $\kbar[\widetilde U]^*=\kbar^*$ and $\Br(\widetilde U_{\kbar})=0$; 
		\item If $n \geq 5$, then $\Pic(\widetilde U_{\bar k})=0$.
		If $n=4$, then $\Pic(\widetilde U_{\bar k})\simeq \ZZ^3$ (ignoring Galois actions).
		\item If $n= 4$, then 
		$\Br(\widetilde U)/\Br_0(\widetilde U) \cong \Ker[H^2(k,\ZZ)\to H^2(k,\ZZ(C(\kbar)))]$ which is finite. 
		If $n\geq 5$, then $\Br(\widetilde U)=\Br_0(\widetilde U)$.  
	\end{enumerate}
	
	Let $k$ be a number field and  $v_0$  a place of $k$. Suppose $\{(a,0,\cdots,0):a\in \kbar^*\}\not \subset Z$. 
	\begin{enumerate}[label=(\roman*)]
		\setcounter{enumi}{4}
		\item If $n=4$, then $\widetilde V\setminus Z$ satisfies central strong approximation with algebraic Brauer-Manin obstruction off $v_0$.
		\item If $n\geq 5$, then $\widetilde V\setminus Z$ satisfies central strong approximation off $v_0$.
	\end{enumerate}
\end{prop}
\begin{proof} 
	 Let $\pi: \widetilde V\to \mathbb A^1_k$ be defined by $x_1$-projection, $\pi$ is surjective and its generic fiber $\widetilde V_{\eta}$ is isomorphic to an affine quadric of dimension $n-2$ over the function field $K$ of $\mathbb A^1_k$. It implies that $\widetilde V$ is geometrically integral with the singular locus $\{(a,0,\cdots,0):a\in \kbar^*\}$, then we proved (i). 
	
	To prove (ii) and (iii), first we assume that $k$ is algebraic closed. By Tsen's theorem and \cite[5.3 and 5.6]{CTX09}, we have $\Br(\widetilde U_{\eta})=\Br(K)=0$, hence $\Br(\widetilde U)=0$. Suppose $n\geq 5$, any fiber of $\pi$ over a point of ${\widetilde U}$ is integral;  by \cite[Proposition 3.2]{CTnote} and \cite[5.3]{CTX09},   we have 
	$$ \kbar[\widetilde U]^*=\kbar^* \text{ and } \Pic(\widetilde U)=0. $$
	Suppose $n=4$. Let $\widetilde U_1$ be the open subset of $\widetilde U$  defined by $x_1\neq 0$. Any fiber of $\pi_{\widetilde U_1}$ is integral, by \cite[Proposition 3.2]{CTnote}  and \cite[5.3]{CTX09},   we have 
	$$ k[\widetilde U_1]^*/k^*\cong \ZZ \text{ and } \Pic(\widetilde U_1)=0. $$
	The inclusion $\widetilde U_1\subset \widetilde U$ derives the exact sequence 
	\begin{equation}\label{seq:exact-pic}
		0\to k[\widetilde U]^*/k^*\to \ZZ \to  \ZZ^4 \to \Pic(\widetilde U)\to\Pic(\widetilde U_1)=0,
	\end{equation}
	the third map is given by sending 1 to $(1,1,1,1)$. Therefore, we have 
	$$k[\widetilde U]^*=k^*\text{ and } \Pic(\widetilde U)\cong \ZZ^3.$$ 
	
	If $n\geq 5$, (ii) and (iii) imply that $\Br(U)=\Br_0(U)$. If $n=4$, the exact sequence (\ref{seq:exact-pic}) implies the case $n=4$ of (iv).

	 The projective variety defined by $\Psi_1=\Psi_2=0$ has the unique singular point $(1:0:\cdots:0)$. Suppose $n=4$, then $V$ is not an Iskovskih surface, 
	 the proof of case   (v) then follows from Theorem \ref{thm:SA-general-type}.  
	
	We shall now prove that the assertion of~(vi) in fact holds for all $n \geq 4$,
	by induction on~$n$. In view of~(v), this will establish~(vi).
	
	Let $n \geq 5$ be such that the assertion of~(vi)
	holds for smaller values of~$n$.
	In the subspace $\mathbb P^{n-2}_k$ (variables $(x_2:\cdots:x_n)$) in $\mathbb P^n_k$, by Bertini's theorem (\cite[Theorem II.8.18]{Hartshorne}), we may choose a codimension~$2$ projective linear subspace 
	$D$  such that $D\cap C$ is nonsingular and that $\widetilde D \cap Z$ has
	codimension~$\geq 1$ in~$Z$, where $\widetilde D$ is the affine linear subspace associated to $D$. In fact, $D\subset \mathbb P^n_k$ can be defined by $l_1(x_2,\cdots,x_n)=l_2(x_2,\cdots, x_n)=0$, where $l_1$ and $ l_2$ are two distinct linear forms in variables $(x_2,\cdots,x_n)$.
	Write $\Lambda$ for the projective line parametrising
	hyperplanes in~$\mathbb A^{n+1}_k$ containing the affine space $\widetilde D$ defined by $l_1=l_2=0$.
	
	Let $g:\widetilde V' \to \widetilde V$ be the blow-up of~$\widetilde V$ along
	$\widetilde D \cap \widetilde V$
	and $f:\widetilde V' \to \Lambda$ the natural morphism corresponding to the blow-up $g$.
	The fibers of~$f$ are the varieties
	$(\widetilde V'\setminus Z)\cap H$
	where $H$ ranges over the planes  of~$\mathbb A^{n+1}_k$ of dimension $n$ containing~$D$. Let  $Z'=g^{-1}(Z)$
	and $\widetilde U'=g^{-1}(\widetilde V^{\sm})\setminus Z'$, and $g^{-1}(\widetilde V^{\sm})$ is nonsingular (which is a blow-up of a nonsingular variety along a nonsingular center).
	The geometric generic fiber~$\widetilde U'_{\etabar}$
	of $f|_{\widetilde U'}:\widetilde U' \to \Lambda$
	is the smooth locus of a variety of the form (\ref{eq:general})
	(with~$n$ replaced by~$n-1$) by removing a closed subset of codimension $2$, in fact, the closed subset is the intersection of $Z$ with the generic plane defined by $\lambda l_1+\mu l_2=0, (\lambda:\mu)\in \Lambda$, hence it also has codimension $2$.
	In particular, by (ii) and (iii), $\widetilde U'_{\etabar}$ has no non-constant invertible function, and the abelian group
	$\Pic(\widetilde U'_{\etabar})$ is torsion-free,
	so that
	$H^1_{\text{ét}}(\widetilde U'_{\bar\eta},\QQ/\ZZ)=0$;
	and that $\Br(\widetilde U'_{\etabar})=0$. Since $\widetilde{U}'$ is nonsingular, 
	we can therefore apply  \cite[Corollary 4.7]{HWW22} to~$f|_{\widetilde U'}$ (recalling that $\Lambda \simeq \mathbb P^1_k$). Any quadratic form in the pencil $\lambda \Psi_1 + \mu \Psi_2$ ($\lambda,\mu \in \kbar $) has rank $\geq n$, 
	the parameters~$\pi_+$ which appear in the statement of Corollary 4.7
	satisfy $L_{m}=k(m)$ for all $m \in M$, so that Conjecture~$F_+$ holds for them
	by \cite[Corollary 6.2 (i)]{HWW22}.
	
	Since $\Br(\widetilde U)/ \Br_0(\widetilde U)$ is finite, we only need to show that arbitrary $(p_v)_v\in\widetilde U(\mathbb A_k)^{\Br(\widetilde U)}$ can be approximated by rational points of $\widetilde U$. By the implicit function theorem, we may assume that $(p_v)_v \in (\widetilde U\setminus \widetilde D)(\mathbb A_k)$. Recall $\widetilde{U}'=g^{-1}(\widetilde V^{\sm})\setminus Z'$. We may assume that $(p_v)_v\in \widetilde{U}'(\mathbb A_k)^{\Br(\widetilde U')}$ since $\Br(\widetilde U')=\Br(\widetilde U)$. 
	By \cite[Corollary 4.7]{HWW22},  $(p_v)_v$ can be approximated arbitrarily well
	by a point $(p'_v)_v\in \widetilde U'^{\sm}_c(\A_k)^{\Br(\widetilde U'^{\sm}_c)}$ for a rational point~$c$ of an arbitrary dense open
	subset of~$\Lambda$.  By the induction hypothesis, $(p'_v)_v \in \widetilde U'^{\sm}_c(\A_k)^{\Br(\widetilde U'^{\sm}_c)}$
	can in turn be approximated, for the adelic topology of $\widetilde V'_c$ (of course for the adelic topology of $\widetilde V'$) off~$v_0$, by a rational point of~$U'_c$. Hence $(p_v)_v$ can be approximated, for the adelic topology of $\widetilde V$ off~$v_0$, by a rational point. 
\end{proof}

\section{Proof of Main Theorem}

Let $V$ be a non-conical geometrically integral complete intersection of two quadrics in $\mathbb P^n$, $n\geq 5$.
Such an intersection $V$ is given by a system of equations
\begin{equation}\label{eq:intesection}
	\begin{cases}
		Q_1(x_0,\cdots, x_n)=0\\
		Q_2(x_0,\cdots, x_n)=0,
	\end{cases}	
\end{equation}
where $Q_1$ and $Q_2$ are two quadratic forms with coefficients in $k$.
Let $\widetilde V$ be the affine cone of $V$ in $\mathbb A^{n+1}_k\setminus (0,\cdots,0)$ defined by equations (\ref{eq:intesection}). 

For any $x \in V^{\sm}(\kbar)$, by a linear change of coordinates, we may assume $x=(1:0:\cdots:0)$, then $V$ can be defined by 
\begin{equation*}
	\begin{cases}
		x_0x_1+Q'_1(x_1,x_2,\cdots, x_n)=0\\
		x_0x_2+Q'_2(x_1,x_2,\cdots, x_n)=0.
	\end{cases}	
\end{equation*} 
Let $C(x)\subset \mathbb P^{n-3}$ be the variety defined by $$Q'_1(0,0,x_3,\cdots, x_n)=Q'_2(0,0,x_3,\cdots, x_n)=0.$$ 

\begin{thm}\label{thm:main}
	Let $k$ be a field of characteristic $0$.  Let $V\subset \mathbb P^n_k, n\geq 5$ be the pure geometrically integral intersection of two quadratics defined by quadratic forms  $Q_1$ and $Q_2$
	which is not a cone. Let $\widetilde V$ be the affine cone of $V$ and  $Z\subset \widetilde V^{\sm}$ is a closed subset of codimension at least $2$. Let $\widetilde U=\widetilde V^{\sm}\setminus Z$. 
	
	We assume: \begin{itemize} \item[a)] there exists a quadratic form of rank $n+1$ in the pencil $\lambda Q_1+\mu Q_2$ $(\lambda,\mu \in \kbar)$;
		\item[b)] $V^{\sm}(k)\neq \emptyset$; 
		\item[c)] $\{x\in V^{\sm}(\kbar): C(x)\text{ is a nonsingular complete intersection}\}\neq \emptyset$.
	\end{itemize}  
	Then 
	\begin{enumerate}[label=(\roman*)]
		\item If  $n\geq 6$, then $\Br(\widetilde U)=\Br_0(\widetilde U)$;
		if $n= 5$, then $\Br(\widetilde U)/\Br_0(\widetilde U)$ is finite.  
	\end{enumerate}
	
	Assume that $k$ is a number field, $v_0$ is a place of $k$.
	\begin{enumerate}[label=(\roman*)]
		\setcounter{enumi}{1}
		\item Suppose $n=5$, then $\widetilde U$ satisfies strong approximation with Brauer-Manin obstruction off $v_0$.
		\item Suppose $n\geq 6$, then $\widetilde U$ satisfies strong approximation off $v_0$.
	\end{enumerate}
\end{thm}

\begin{remark}\label{rem:main}\begin{enumerate}
		\item[1)]  If $n\geq 7$ and $\prod_v V^{\sm}(k_v)\neq \emptyset$,  then $V^{\sm}(k)\neq \emptyset$ by 
		\cite[Theorem 3.1]{CTSSD1}, \cite[Theorem 1.1]{HB18}, and \cite[Theorem 1]{Mol}. 
		\item[2)] If $n\geq 5$ and $V\subset \mathbb P^n_k$ is smooth, then $$\{x\in V^{\sm}(\kbar): C(x)\text{ is a nonsingular complete intersection}\}$$ is not empty (see \cite[Proposition 2.14 (f)]{CT23}, which asserts that for a smooth intersection of two quadrics, there exists a point whose associated residual intersection is also smooth).
		
		\item[3)] Condition $c)$ implies the rank of any quadratic form in the pencil $\lambda Q_1+\mu Q_2$ is at least~$n-1$. Indeed, for any quadratic form $\psi$ in the pencil, consider the restriction $\psi(x_0,0,0,x_3,\cdots,x_n)$, which is a quadratic form in the variables $x_0, x_3, \ldots, x_n$. By assumption $c)$, for some $x$, the intersection $C(x)$ defined by the restrictions of $Q_1$ and $Q_2$ is a nonsingular complete intersection, which means that for any nontrivial linear combination $\psi$, the restricted form has rank at least $n-3$. Since the original form $\psi$ depends also on $x_1$ and $x_2$, and the restriction only sets $x_1 = x_2 = 0$, the rank of $\psi$ must be at least $(n-3) + 2 = n-1$. Thus, every quadratic form in the pencil has rank at least $n-1$.
		
		\item[4)] Possibly the least rank in the family satisfying the conditions $a)-c)$ is $n-1$. For example, when $n=5$,
		$$
		\begin{cases}
			x_0x_1 + x_4^2 + x_5^2 = 0, \\
			x_0x_2 + x_1^2 + x_2^2 + x_3^2 + x_4^2 - x_5^2 = 0.
		\end{cases}
		$$
		Here, the least rank is $n-1=4$, which satisfies the conditions $a)-c)$.
	\end{enumerate}
\end{remark}

\begin{proof}
We begin by outlining the main strategy of the proof. The goal is to establish strong approximation (with or without Brauer-Manin obstruction, depending on $n$) for the smooth locus of the affine cone over a complete intersection of two quadrics, possibly with a closed subset of codimension at least $2$ removed. The proof proceeds by first ensuring the Zariski density of rational points, then selecting a suitable rational point to simplify the equations, and finally applying the fibration method and induction on the dimension, together with known results on the Brauer group and Picard group, to deduce the desired approximation property.

When $V^{\sm}(k) \neq \emptyset$, the set of rational points $V^{\sm}(k)$ is Zariski dense. In fact, if $n\geq 6$, by \cite[Theorem 3.11]{CTSSD1}, $V^{\sm}$ satisfies weak approximation. If $n=5$, the Brauer-Manin obstruction to weak approximation is the only one for any smooth projective model of $V$ by \cite[Theorem 1]{CTS92}. 

	For any $x\in V^{\sm}(\kbar)$, denote the tangent space at $x$ to $V$ by $T_{x}$. 
	\begin{lemma}\label{lem:point}
		There exists a non-empty open subset $U\subset V^{\sm}(\kbar)$ such that for any $x\in U$,  $T_{x,V}\cap Z\subset Z$ has codimension at least 1. 
	\end{lemma}
	\begin{proof} We choose a point on each irreducible components of $Z$, then we get a finite subset $\widetilde M\subset Z$. Let $M$ be the image of $\widetilde M$ in $V$. Let $U:=\{x \in V^{\sm}(\kbar)\mid T_{x,V} \cap M=\emptyset\}$ be the open subset. To show that $U$ is nonempty, it is enough to show, for any given $m\in M$,  it is impossible that $m\in T_{x,V}$ for any $x \in V^{\sm}(\kbar)$.  
		
		For any point $x\in V^{\sm}(\kbar)$, we assume that the tangent space $ T_{x,V}$ contains~$m$. Since $V$ is non-conical, there exists a quadric $Q\supset V$ such that $m$ is not a vertex of $Q$. Then the points $x\in V^{\sm}(\kbar)$ such that $m\in T_{x,V}$ are contained in a hyperplane $H\subset \mathbb P^n$ which is the space orthogonal to $m$ with respect to the quadratic form defining $Q$. In particular, $V$ is contained in $H$ and this contradicts \cite[Lemma 1.3]{CTSSD1} which ensures that $V(\kbar)$ generates $\mathbb P^n$. 
	\end{proof}
	
	For any point $x\in V(\bar k)$, $C(x)$ is nonsingular complete intersection if and only if the determinant of the pencil of $C(x)$ is of degree $n-2$ and just has $n-2$ distinct roots by \cite[Proposition 2.1]{Rei72}, which is an open condition. Since $V^{\sm}(k)$ is Zariski dense and condition $c)$ is an open condition, we may choose a rational point $p\in V^{\sm}(k)$, such that  $C(p)$ is a nonsingular complete intersection and   $T_p\cap Z\subset Z$  has codimension at least $1$ by Lemma \ref{lem:point}. After a linear change of coordinates and a replacement of $Q_1$ and  $Q_2$ by suitable linear combinations, we may assume $p= (1: 0: \cdots: 0)$ and
	\begin{equation}\label{eq: V}\begin{cases}
			Q_1 = x_0x_1 + Q'_1(x_1,\cdots,x_n)\\
			Q_2 = x_0x_2 + Q'_2(x_1,\cdots, x_n),
		\end{cases}
	\end{equation}
	where $Q'_1$ and $Q'_2$ are quadratic forms in $(x_1,\cdots,x_n)$. Therefore, the affine cone $\widetilde V\subset \mathbb {A}^{n+1}\setminus {(0,\cdots,0)}$ is defined by $Q_1=Q_2=0$ with $Q_1,Q_2$  in (\ref{eq: V}).
	
	Let $D$ be the closed subset of $\widetilde V$ defined by $x_1=x_2=0$, in fact, $D$ is the cone over $C(p)$  with vertex $p$, and $D\subset \widetilde V^{\sm}$ by condition $c)$. Let $g: \widetilde V'\to \widetilde V$ be the blow-up of $\widetilde V$ along~$D$. Then  $\widetilde V'$ is defined by 
	\begin{equation*}
		\begin{cases}
			&x_0x_1+Q'_1(x_1,\cdots,x_n)=0\\
			&x_0x_2+Q'_2(x_1,\cdots, x_n)=0\\
			&sx_1+tx_2=0,
		\end{cases}
	\end{equation*}
	where $[s:t]\in \mathbb P^1$. 
	Let $f: \widetilde V'\to \mathbb P^1$ be the projection morphism $(x_0,\cdots,x_n;[s:t])\mapsto [s:t]$. Let  $Z'=g^{-1}(Z)$
	and $\widetilde U'=g^{-1}(\widetilde V^{\sm})\setminus Z'$.
	As~$Z'\cap \widetilde U'$ has codimension~$\geq 2$ in~$\widetilde U'$, the generic fiber ~$\widetilde U'_{\etabar}$ of $f|_{\widetilde U'}:\widetilde U' \to \mathbb P^1_k$ is the affine variety removing a closed subset of codimension $\geq 2$ defined by   
	\begin{equation}\label{eq: fiber}
		\begin{cases}
			&x_0x_1+Q'_1(x_1,-(s/t) x_1,\cdots,x_n)=0\\
			&(s/t) Q'_1(x_1,-(s/t) x_1,\cdots,x_n)+ Q'_2(x_1,-(s/t) x_1,\cdots, x_n)=0,
		\end{cases}
	\end{equation} which has the form (\ref{eq:general}) (here we replaced $n$ by~$n-1$). 
	
	
	Let $H:=\{(a,0,\cdots,0):a\in \kbar\}$  be the singular locus of $\widetilde V$ of codimension $\geq 2$. Thus $D\setminus H$ is nonsingular and so is $g^{-1}(\widetilde V^\sm\setminus H)$. In fact, $g^{-1}(\widetilde U\setminus H)$ is  the smooth locus $\widetilde U'^\sm$ of $\widetilde U'$.  By \cite[Proposition 3.7.10]{ctskbook}, $\Br(\widetilde U'^\sm)=\Br(\widetilde U\setminus H)=\Br(\widetilde U^{\sm})$. 
	
	First we  assume that $k$ is algebraic closed. By Tsen's theorem, $\Br(k(\eta))=0$. If $n\geq 6$, by Proposition \ref{prop:genral} $(iv)$, we have $\Br(\widetilde U'^{\sm}_{\eta})=0$, hence $\Br(\widetilde U'^{\sm}_{\kbar})=0$. If $n=5$, the $C(\kbar)$ in Proposition \ref{prop:genral} $(iv)$ is the union of $4$ distinct points, 
	then $\Br(\widetilde U'^{\sm}_{\eta})$ is finite, hence $\Br(\widetilde U'^{\sm}_{\kbar})$ is finite.
	
	Suppose $n\geq 6$.  In the pencil of (\ref{eq: fiber}), there exists a form of rank $n\geq 6$ by condition a),  any form has rank $n-3\geq 3$ by Remark \ref{rem:main} 3), therefore all fibers are geometrically integral by \cite[Lemma 1.11]{CTSSD1}.  By \cite[Proposition 3.2]{CTnote} and Proposition \ref{prop:genral} $(ii)+(iii)$,   we have 
	$$ \kbar[\widetilde U'^{\sm}]^*=\kbar^* \text{ and } \Pic(\widetilde U'^{\sm}_\kbar)=\mathbb Z. $$
	By Hochschild-Serre's spectral sequence, we have  $\Br(\widetilde U'^\sm)=\Br_1(\widetilde U'^\sm)=0$.
	Suppose $n=5$.  By Remark \ref{rem:main} $3)$, any form in the pencil of (\ref{eq: fiber}) has rank $\geq2$ and there exists a form of rank $5$ by condition a), so there are at most $3$ fibers which are not geometrically integral of multiplicity $1$ by \cite[Lemma 1.10 and 1.11]{CTSSD1}. By \cite[Proposition 3.2]{CTnote} and Proposition \ref{prop:genral} $(ii)+(iii)$,  we have 
	$$ \kbar[\widetilde U'^{\sm}]^*=\kbar^* \text{ and } \Pic(\widetilde U'^{\sm}_\kbar) \text{ is finitely generated  and torsion-free}.$$
	Therefore, $\Br(\widetilde U'^{\sm})/\Br_0(\widetilde U')$ is finite.
	
	We shall now prove that the assertion of~$(ii)$ in fact holds for all $n \geq 4$,
	by induction on~$n$. In view of~$(i)$, this will establish~$(iii)$.
	
	By Proposition \ref{prop:genral} $(ii)+(iii)$,  the generic fiber ~$\widetilde U'^{\sm}_{\etabar}$ has no non-constant invertible function, and the abelian group
	$\Pic(\widetilde U'^{\sm}_{\etabar})$ is torsion-free,
	so that
	$H^1_{\text{ét}}(\widetilde U'^{\sm}_{\bar\eta},\QQ/\ZZ)=0$;
	and that $\Br(\widetilde U'^{\sm}_{\etabar})=0$. 
	By the above discussion, if $n\geq 6$,  all fibers are geometrically integral; if $n=5$, there are at most $3$ fibers which are non-split and any non-split fiber splits by a quadratic extension. By \cite[Corollary 6.2 (i)+(ii)]{HWW22}, we can apply \cite[Corollary 4.7]{HWW22}.
	For any $(p_v)_v\in\widetilde U(\mathbb A_k)^{\Br(\widetilde U)}$, by the implicit function theorem, we may assume that $(p_v)_v \in (\widetilde U\setminus \widetilde D)(\mathbb A_k)$. Recall $\widetilde{U}'=g^{-1}(\widetilde V^{\sm})\setminus Z'$. Then we may assume that $(p_v)_v\in \widetilde{U}'^{\sm}(\mathbb A_k)^{\Br(\widetilde U'^{\sm})}$ by Harari's formal lemma. 
	By \cite[Corollary 4.7]{HWW22},  $(p_v)_v$ can be approximated arbitrarily well
	by a point $(p'_v)_v\in \widetilde U'^{\sm}_c(\A_k)^{\Br(\widetilde U'^{\sm}_c)}$ for a rational point~$c$ of an arbitrary dense open
	subset of~$\mathbb P^1$.  By Proposition \ref{prop:genral}, $(p'_v)_v \in \widetilde U'^{\sm}_c(\A_k)^{\Br(\widetilde U'^{\sm}_c)}$
	can be approximated, 
	for the adelic topology of $\widetilde U'_c$ off~$v_0$, by a rational point of~$U'_c$. Hence $(p_v)_v$ can be approximated, for the adelic topology of $\widetilde V^{\sm}$ off~$v_0$, by a rational point.
\end{proof}

Now we prove the main theorem of the paper. It suffices to show that the three conditions in Theorem \ref{thm:main} hold on $V$. The following two lemmas will be used to verify conditions a) and c).

\begin{lemma}\label{lem:det} 
	Let $k$ be a field of characteristic $0$.
	Let $V\subset \mathbb P^n_k$ be a pure geometrically integral intersection of two quadrics $Q_1=Q_2=0$ which is not a cone. If the homogeneous
	polynomial $P(\lambda, \mu ) = \det (\lambda Q_1 +\mu Q_1 )$ vanishes identically, then  $V$ has infinitely many  singular $k$-rational points.
\end{lemma}
\begin{proof}
By changing coordinates, we may assume that $Q_1=\sum_{i=0}^r a_i x_i^2$  with $a_i \in k^*$ and
	$ 1\leq r <n $. Let $[b_{ij} ]$ be the matrix of $Q_2$. The coefficient of $\lambda^{r}\mu^{n-r}$ of $P(\lambda, \mu )$ is $a_0\cdots a_r \det(B)$, where $B$ is the $(n-r, n-r)$-matrix with entries
	$(b_{ij})_{r+1\leq i,j\leq n}$. Thus $\det B=0$, and there exists a non-zero vector $z=(c_{r+1},\cdots, c_n)$
	with coordinates in $k$ such that $B z^t=0$. The point $z=[0:\cdots:0: c_{r+1}:\cdots:c_n]\in V(k)$ and it is clearly conical on the quadric $Q_1=0$, hence it is a singular $k$-point. By replacing $Q_1$ by any quadric $\lambda Q_1+\mu Q_2$ and repeating the above discussion, we get a another singular $k$-point $z_{[\lambda:\mu]}$ which is conical on the quadric $\lambda Q_1+\mu Q_2=0$. The proof now follows from that $z_{[\lambda:\mu]}$ are pairwise distinct, otherwise some $z_{[\lambda:\mu]}$ is conical on $V$ which contradicts to that $V$ is not a cone. 
\end{proof}

\begin{lemma}\label{lem:fano-nonsingular} 	Let $k$ be a field of characteristic $0$.
	Let $V\subset \mathbb P^n_k, n\geq 5$ be a pure geometrically integral intersection of two quadrics which is not a cone. Suppose that $V_\kbar$ only has finite singular points. Then, 
	there exists a nonempty open subset $U\subset V^{\sm}$ such that, for any $x\in U(\kbar)$, $C(x)$ is a nonsingular complete intersection, $i.e.$, condition $c)$ of Theorem \ref{thm:main} holds on $V$.
\end{lemma}
\begin{proof} It suffices to prove the statement over $\kbar$.  
	For any singular point $p\in V^{\text{sing}}(\kbar)$, let $T(p):=\{x\in V^{\sm}(\kbar)\mid p\in T_{x,V}\}$ be the closed subset of $V$. Since $V$ is not a cone, the open subset $U_0:=V \setminus  \cup_{p\in V^{\text{sing}}(\kbar)}T(p)$ is nonempty. 
	
	For any smooth point $x\in U_0$, since $\dim(T_{x,V}\cap V)\geq n-4\geq 1$  by \cite[Theorem~7.2]{Hartshorne},  there is a line on $V_{sm}$ passing through $x$. In particular, the evaluation morphism of Hom scheme of lines $$ev:\mathbb{P}^1\times \Hom(\mathbb{P}^1,V)\to V$$ is dominant. Since Hom scheme of lines is of finite type, by the same argument as in \cite[Theorem II.3.11]{Kollar}, every line through a general point on $U_0$ is free on $V^{sm}$. 
	In particular, the Fano scheme of lines $C(x)$ passing through a general point $x$ is smooth. 
\end{proof}

\begin{proof}[Proof of Theorem \ref{thm:main'}]

	By Lemma \ref{lem:det} and \ref{lem:fano-nonsingular}, conditions (a)+(c) in Theorem \ref{thm:main} hold on $V$.  The proof of case (i) and (ii) then follows from Theorem \ref{thm:main}. For case (iii), we may assume $\widetilde U(\A_k)\neq \emptyset$, by the natural projection, we have  $V^{\sm}(\A_k)\neq \emptyset$. Therefore $V^{\sm}(k)\neq \emptyset$ by 
	\cite[Theorem 3.1]{CTSSD1}, \cite[Theorem 1.1]{HB18}, and \cite[Theorem 1]{Mol} and the proof follows.
\end{proof}

\begin{remark}
	When $n=5 \text{ or } 6$, we may expect that $V^\sm$ satisfies weak approximation with Brauer-Manin obstruction, $i.e.$,  the condition $V^\sm(k)\neq \emptyset$ should be removed (similarly as $n\geq 7$). However, this assertion is still open.
\end{remark}




\bibliography{myref-2}	

@article {Sko96,
	AUTHOR = {Skorobogatov, A. N.},
	TITLE = {Descent on fibrations over the projective line},
	JOURNAL = {Amer. J. Math.},
	FJOURNAL = {American Journal of Mathematics},
	VOLUME = {118},
	YEAR = {1996},
	NUMBER = {5},
	PAGES = {905--923},
	ISSN = {0002-9327,1080-6377},
	MRCLASS = {11G35 (14G25)},
	MRNUMBER = {1408492},
	MRREVIEWER = {J\"org\ Jahnel},
	URL =
	{http://muse.jhu.edu/journals/american_journal_of_mathematics/v118/118.5skorobogatov.pdf},
}

@article{BS19,
	AUTHOR = {Browning, T. D. and Schindler, D.},
	TITLE = {Strong approximation and a conjecture of {H}arpaz and
	{W}ittenberg},
	JOURNAL = {Int. Math. Res. Not.},
	FJOURNAL = {International Mathematics Research Notices. IMRN},
	YEAR = {2019},
	VOLUME = {2019},
	NUMBER = {14},
	PAGES = {4340--4369},
	ISSN = {1073-7928,1687-0247},
	MRCLASS = {14G05 (11D57 11N56 14D10 14F22)},
	MRNUMBER = {3984071},
	MRREVIEWER = {Sho\ Tanimoto},
	DOI = {10.1093/imrn/rnx252},
	URL = {https://doi.org/10.1093/imrn/rnx252},
}

@incollection {CT90,
	AUTHOR = {J.-L. Colliot-Th\'{e}l\`ene},
	TITLE = {Surfaces rationnelles fibr\'{e}es en coniques de degr\'{e}
	{$4$}},
	BOOKTITLE = {S\'{e}minaire de {T}h\'{e}orie des {N}ombres, {P}aris
	1988--1989},
	SERIES = {Progr. Math.},
	VOLUME = {91},
	PAGES = {43--55},
	PUBLISHER = {Birkh\"{a}user Boston, Boston, MA},
	YEAR = {1990},
	ISBN = {0-8176-3493-2},
	MRCLASS = {14G25 (11G35 14J26)},
	MRNUMBER = {1104699},
	MRREVIEWER = {Wayne\ Raskind},
}

@article {CorT88,
	AUTHOR = { D. F. Coray and M. A. Tsfasman},
	TITLE = {Arithmetic on singular {D}el {P}ezzo surfaces},
	JOURNAL = {Proc. London Math. Soc. (3)},
	FJOURNAL = {Proceedings of the London Mathematical Society. Third Series},
	VOLUME = {57},
	YEAR = {1988},
	NUMBER = {1},
	PAGES = {25--87},
	ISSN = {0024-6115,1460-244X},
	MRCLASS = {11G35 (11E04 11E95 11G25 14J17 14J20 14J26)},
	MRNUMBER = {940430},
	MRREVIEWER = {Noriko\ Yui},
	DOI = {10.1112/plms/s3-57.1.25},
	URL = {https://doi.org/10.1112/plms/s3-57.1.25},
}

@book{Rei72,
	author = {M. Reid},
	publisher = {Thesis, Trinity College,
	Cambridge},
	title = {The complete intersection of two or more quadrics},
	year = {June 1972}}

@article {Har94,
	AUTHOR = {D. Harari},
	TITLE = {M\'{e}thode des fibrations et obstruction de {M}anin},
	JOURNAL = {Duke Math. J.},
	FJOURNAL = {Duke Mathematical Journal},
	VOLUME = {75},
	YEAR = {1994},
	NUMBER = {1},
	PAGES = {221--260},
	ISSN = {0012-7094,1547-7398},
	MRCLASS = {11G35 (14G25 14M10)},
	MRNUMBER = {1284820},
	MRREVIEWER = {Takeshi\ Ooe},
	DOI = {10.1215/S0012-7094-94-07507-8},
	URL = {https://doi.org/10.1215/S0012-7094-94-07507-8},
}

@article {CTnote,
	AUTHOR = {J.-L. Colliot-Th\'{e}l\`ene},
	TITLE = {Lectures on linear algebraic groups},
	JOURNAL = {https://www.imo.universite-paris-saclay.fr/~jean-louis.colliot-thelene/BeijingLectures2Juin07.pdf},
	FJOURNAL = {},
	VOLUME = {},
	YEAR = {2007},
	NUMBER = {},
	PAGES = {},
	ISSN = {},
	MRCLASS = {},
	MRNUMBER = {},
}

@article {SX,
	AUTHOR = {H. Song and F. Xu},
	TITLE = {Strong approximation with {B}rauer-{M}anin obstruction for
	certain singular varieties by explicit blowing up},
	JOURNAL = {Acta Math. Sinica (Chinese Ser.)},
	FJOURNAL = {Acta Mathematica Sinica. Chinese Series},
	VOLUME = {67},
	YEAR = {2024},
	NUMBER = {2},
	PAGES = {393--405},
	ISSN = {0583-1431},
	MRCLASS = {14G05 (11D09 11D57 11G05 14F22)},
	MRNUMBER = {4715018},
}

@article {HB18,
	AUTHOR = {D. R. Heath-Brown},
	TITLE = {Zeros of pairs of quadratic forms},
	JOURNAL = {J. Reine Angew. Math.},
	FJOURNAL = {Journal f\"{u}r die Reine und Angewandte Mathematik. [Crelle's
	Journal]},
	VOLUME = {739},
	YEAR = {2018},
	PAGES = {41--80},
	ISSN = {0075-4102,1435-5345},
	MRCLASS = {11E12 (14F22 14G05)},
	MRNUMBER = {3808257},
	MRREVIEWER = {John\ M.\ Voight},
	DOI = {10.1515/crelle-2015-0062},
	URL = {https://doi.org/10.1515/crelle-2015-0062},
}

@article {Mol,
	AUTHOR = {A. Molyakov},
	TITLE = {Le principe de Hasse pour les intersections de deux quadriques dans $\mathbb {P}^7$ },
	JOURNAL = {arXiv:2305.00313v3},
	FJOURNAL = {},
	VOLUME = {},
	YEAR = {2023},
	NUMBER = {},
	PAGES = {},
	ISSN = {},
	MRCLASS = {11G35 (11G99 14M10)},
	MRNUMBER = {1150818},
	MRREVIEWER = {},
}

@article{weitorus,
	author = {D. Wei},
	journal = {Acta Math. Sinica (English Ser.)},
	number = {1},
	pages = {95--103},
	publisher = {Springer},
	title = {Strong approximation for a toric variety},
	volume = {37},
	year = {2021}}

@article {CTS92,
	AUTHOR = {J.-L. Colliot-Th\'{e}l\`ene and A. N. Skorobogatov},
	TITLE = {Approximation faible pour les intersections de deux quadriques
	en dimension {$3$}},
	JOURNAL = {C. R. Acad. Sci. Paris S\'{e}r. I Math.},
	FJOURNAL = {Comptes Rendus de l'Acad\'{e}mie des Sciences. S\'{e}rie I.
	Math\'{e}matique},
	VOLUME = {314},
	YEAR = {1992},
	NUMBER = {2},
	PAGES = {127--132},
	ISSN = {0764-4442},
	MRCLASS = {11G35 (11G99 14M10)},
	MRNUMBER = {1150818},
	MRREVIEWER = {Noriko\ Yui},
}

@article {langweil,
	AUTHOR = {Lang, S. and Weil, A.},
	TITLE = {Number of points of varieties in finite fields},
	JOURNAL = {Amer. J. Math.},
	FJOURNAL = {American Journal of Mathematics},
	VOLUME = {76},
	YEAR = {1954},
	PAGES = {819--827},
	ISSN = {0002-9327,1080-6377},
	MRCLASS = {14.0X},
	MRNUMBER = {65218},
	MRREVIEWER = {B.\ Segre},
	DOI = {10.2307/2372655},
	URL = {https://doi.org/10.2307/2372655},
}

@incollection{HS13,
	author = {D. Harari  and A. N. Skorobogatov},
	booktitle = {Torsors, {\'e}tale homotopy and applications to rational points},
	pages = {250--279},
	publisher = {London Mathematical Society Lecture Note Series 405},
	title = {Descent theory for open varieties},
	year = {2013}}

@article{CTX13,
	author = {J.-L. Colliot-Th{\'e}l{\`e}ne  and F. Xu},
	journal = {Acta Arith.},
	pages = {169--199},
	publisher = {Instytut Matematyczny Polskiej Akademii Nauk},
	title = {Strong approximation for the total space of certain quadric fibrations},
	volume = {157},
	year = {2013}}

@article{CT23,
	author = {J.-L. Colliot-Th{\'e}l{\`e}ne},
	journal = { J. reine angew. Math. (to appear) },
	pages = {},
	publisher = {},
	title = {Retour sur l'arithmétique des intersections de deux quadriques, avec un appendice par A. Kuznetsov},
	volume = {},
	year = {2023}}

@incollection{CT03,
	author = {J.-L. Colliot-Th{\'e}lene},
	booktitle = {Higher dimensional varieties and rational points},
	pages = {171--221},
	publisher = {Springer},
	title = {Points rationnels sur les fibrations},
	year = {2003}}

@book {ctskbook,
	AUTHOR = {J.-L. Colliot-Th\'el\`ene and A. N. Skorobogatov},
	TITLE = {The {B}rauer-{G}rothendieck group},
	SERIES = {Ergebnisse der Mathematik und ihrer Grenzgebiete. 3. Folge. A
	Series of Modern Surveys in Mathematics [Results in
	Mathematics and Related Areas. 3rd Series. A Series of Modern
	Surveys in Mathematics]},
	VOLUME = {71},
	PUBLISHER = {Springer, Cham},
	YEAR = {[2021] \copyright 2021},
	PAGES = {xv+453},
	ISBN = {978-3-030-74247-8; 978-3-030-74248-5},
	MRCLASS = {14F22 (14E08 14G05 14G12 14K05)},
	MRNUMBER = {4304038},
	MRREVIEWER = {Thomas\ Benedict\ Williams},
	DOI = {10.1007/978-3-030-74248-5},
	URL = {https://doi.org/10.1007/978-3-030-74248-5},
}

@article {HWW22,
	AUTHOR = {Harpaz, Y. and Wei, D. and Wittenberg, O.},
	TITLE = {Rational points on fibrations with few non-split fibres},
	JOURNAL = {J. Reine Angew. Math.},
	FJOURNAL = {Journal f\"{u}r die Reine und Angewandte Mathematik. [Crelle's
	Journal]},
	VOLUME = {791},
	YEAR = {2022},
	PAGES = {89--133},
	ISSN = {0075-4102,1435-5345},
	MRCLASS = {14G12 (11G35)},
	MRNUMBER = {4489626},
	MRREVIEWER = {Mikhail\ Borovoi},
	DOI = {10.1515/crelle-2022-0042},
	URL = {https://doi.org/10.1515/crelle-2022-0042},
}

@article {HW,
	AUTHOR = {Harpaz, Y. and Wittenberg, O.},
	TITLE = {On the fibration method for zero-cycles and rational points},
	JOURNAL = {Ann. of Math. (2)},
	FJOURNAL = {Annals of Mathematics. Second Series},
	VOLUME = {183},
	YEAR = {2016},
	NUMBER = {1},
	PAGES = {229--295},
	ISSN = {0003-486X,1939-8980},
	MRCLASS = {14G25 (14C15 14M22)},
	MRNUMBER = {3432584},
	MRREVIEWER = {Amanda\ Knecht},
	DOI = {10.4007/annals.2016.183.1.5},
	URL = {https://doi.org/10.4007/annals.2016.183.1.5},
}

@article {browningmatthiesen,
	AUTHOR = {T. D. Browning  and L. Matthiesen},
	TITLE = {Norm forms for arbitrary number fields as products of linear
	polynomials},
	JOURNAL = {Ann. Sci. \'{E}c. Norm. Sup\'{e}r. (4)},
	FJOURNAL = {Annales Scientifiques de l'\'{E}cole Normale Sup\'{e}rieure.
	Quatri\`eme S\'{e}rie},
	VOLUME = {50},
	YEAR = {2017},
	NUMBER = {6},
	PAGES = {1383--1446},
	ISSN = {0012-9593,1873-2151},
	MRCLASS = {11D57 (14G05)},
	MRNUMBER = {3742196},
	MRREVIEWER = {P.\ Bundschuh},
	DOI = {10.24033/asens.2348},
	URL = {https://doi.org/10.24033/asens.2348},
}

@article {bms,
	AUTHOR = {Browning, T. D. and Matthiesen, L. and Skorobogatov, A. N.},
	TITLE = {Rational points on pencils of conics and quadrics with many
	degenerate fibers},
	JOURNAL = {Ann. of Math. (2)},
	FJOURNAL = {Annals of Mathematics. Second Series},
	VOLUME = {180},
	YEAR = {2014},
	NUMBER = {1},
	PAGES = {381--402},
	ISSN = {0003-486X,1939-8980},
	MRCLASS = {14G05 (14C21 14G25)},
	MRNUMBER = {3194818},
	MRREVIEWER = {Cecilia\ Salgado},
	DOI = {10.4007/annals.2014.180.1.8},
	URL = {https://doi.org/10.4007/annals.2014.180.1.8},
}

@article {derenthalsmeetswei,
	AUTHOR = {Derenthal, U. and Smeets, A. and Wei, D.},
	TITLE = {Universal torsors and values of quadratic polynomials
	represented by norms},
	JOURNAL = {Math. Ann.},
	FJOURNAL = {Mathematische Annalen},
	VOLUME = {361},
	YEAR = {2015},
	NUMBER = {3-4},
	PAGES = {1021--1042},
	ISSN = {0025-5831,1432-1807},
	MRCLASS = {14G05 (11D57 14F22)},
	MRNUMBER = {3319557},
	MRREVIEWER = {Stefan\ Schr\"{o}er},
	DOI = {10.1007/s00208-014-1106-7},
	URL = {https://doi.org/10.1007/s00208-014-1106-7},
}

@article {heathbrownskorobogatov,
	AUTHOR = {Heath-Brown, D. R. and Skorobogatov, A. N.},
	TITLE = {Rational solutions of certain equations involving norms},
	JOURNAL = {Acta Math.},
	FJOURNAL = {Acta Mathematica},
	VOLUME = {189},
	YEAR = {2002},
	NUMBER = {2},
	PAGES = {161--177},
	ISSN = {0001-5962,1871-2509},
	MRCLASS = {14G05 (11D57 11G25 14F22 14G25)},
	MRNUMBER = {1961196},
	MRREVIEWER = {Tam\'{a}s\ Szamuely},
	DOI = {10.1007/BF02392841},
	URL = {https://doi.org/10.1007/BF02392841},
}

@article {CTSa,
	AUTHOR = {J.-L. Colliot-Th{\'e}l{\`e}ne and P. Salberger},
	TITLE = {Arithmetic on some singular cubic hypersurfaces},
	JOURNAL = {Proc. London Math. Soc. (3)},
	FJOURNAL = {Proceedings of the London Mathematical Society. Third Series},
	VOLUME = {58},
	YEAR = {1989},
	NUMBER = {3},
	PAGES = {519--549},
	ISSN = {0024-6115,1460-244X},
	MRCLASS = {11G35 (14G25)},
	MRNUMBER = {988101},
	MRREVIEWER = {Andrew\ Bremner},
	DOI = {10.1112/plms/s3-58.3.519},
	URL = {https://doi.org/10.1112/plms/s3-58.3.519},
}

@article{CTX09,
	author = {J.-L. Colliot-Th{\'e}l{\`e}ne and F. Xu},
	doi = {10.1112/S0010437X0800376X},
	fjournal = {Compositio Mathematica},
	issn = {0010-437X},
	journal = {Compos. Math.},
	mrclass = {11G35 (11D85 11E12 14F22 14G25 20G30)},
	mrnumber = {2501421 (2010c:11072)},
	mrreviewer = {Tam{\'a}s Szamuely},
	note = {With an appendix by Dasheng Wei and Xu},
	number = {2},
	pages = {309--363},
	title = {Brauer-{M}anin obstruction for integral points of homogeneous spaces and representation by integral quadratic forms},
	url = {http://dx.doi.org/10.1112/S0010437X0800376X},
	volume = {145},
	year = {2009}}

@article{CTS87,
	author = {J.-L. Colliot-Th{\'e}l{\`e}ne and J.-J. Sansuc},
	coden = {DUMJAO},
	doi = {10.1215/S0012-7094-87-05420-2},
	fjournal = {Duke Mathematical Journal},
	issn = {0012-7094},
	journal = {Duke Math. J.},
	mrclass = {11G35 (14G25)},
	mrnumber = {899402 (89f:11082)},
	mrreviewer = {Daniel Coray},
	number = {2},
	pages = {375--492},
	title = {La descente sur les vari\'et\'es rationnelles. {II}},
	url = {http://dx.doi.org/10.1215/S0012-7094-87-05420-2},
	volume = {54},
	year = {1987}}

@article {CTSSD1,
	AUTHOR = {Colliot-Th\'{e}l\`ene, J.-L. and Sansuc, J.-J. and
	Swinnerton-Dyer, P.},
	TITLE = {Intersections of two quadrics and {C}h\^{a}telet surfaces.
	{I,II}},
	JOURNAL = {J. Reine Angew. Math.},
	FJOURNAL = {Journal f\"{u}r die Reine und Angewandte Mathematik. [Crelle's
	Journal]},
	VOLUME = {373},
	YEAR = {1987},
	PAGES = {37--107},
	ISSN = {0075-4102,1435-5345},
	MRCLASS = {11G35 (11E81 14J20 14J26)},
	MRNUMBER = {870307},
	MRREVIEWER = {Noriko\ Yui},
}

@book{Hartshorne,
	address = {New York},
	author = {R. Hartshorne},
	isbn = {0-387-90244-9},
	mrclass = {14-01},
	mrnumber = {0463157 (57 \#3116)},
	mrreviewer = {Robert Speiser},
	note = {Graduate Texts in Mathematics, No. 52},
	pages = {xvi+496},
	publisher = {Springer-Verlag},
	title = {Algebraic geometry},
	year = {1977}}

@article {nisnevic,
	AUTHOR = {Nisnevi{\v{c}}, L. B.},
	TITLE = {On the number of points of an algebraic manifold in a prime
	finite field},
	JOURNAL = {Dokl. Akad. Nauk SSSR (N.S.)},
	FJOURNAL = {Doklady Akademii Nauk SSSR},
	VOLUME = {99},
	YEAR = {1954},
	PAGES = {17--20}
}

@book{Kollar,
	address = {Berlin},
	author = {Koll{\'a}r, J.},
	isbn = {3-540-60168-6},
	mrclass = {14-02 (14C05 14E05 14F17 14J45)},
	mrnumber = {1440180 (98c:14001)},
	mrreviewer = {Yuri G. Prokhorov},
	pages = {viii+320},
	publisher = {Springer-Verlag},
	series = {Ergebnisse der Mathematik und ihrer Grenzgebiete. 3. Folge. A Series of Modern Surveys in Mathematics [Results in Mathematics and Related Areas. 3rd Series. A Series of Modern Surveys in Mathematics]},
	title = {Rational curves on algebraic varieties},
	volume = {32},
	year = {1996}}

@book{Serre2002,
	address = {Berlin},
	author = {J.-L. Serre},
	edition = {English},
	isbn = {3-540-42192-0},
	mrclass = {12G05 (11R34)},
	mrnumber = {1867431 (2002i:12004)},
	note = {Translated from the French by Patrick Ion and revised by the author},
	pages = {x+210},
	publisher = {Springer-Verlag},
	series = {Springer Monographs in Mathematics},
	title = {Galois cohomology},
	year = {2002}}

@book{Skobogatov2001,
	address = {Cambridge},
	author = {A. N. Skorobogatov},
	doi = {10.1017/CBO9780511549588},
	isbn = {0-521-80237-7},
	mrclass = {14G05 (11G35 11S25 14D10 14G25 14L30)},
	mrnumber = {1845760 (2002d:14032)},
	mrreviewer = {Tam{\'a}s Szamuely},
	pages = {viii+187},
	publisher = {Cambridge University Press},
	series = {Cambridge Tracts in Mathematics},
	title = {Torsors and rational points},
	url = {http://dx.doi.org/10.1017/CBO9780511549588},
	volume = {144},
	year = {2001}}
\bibliographystyle{alpha}	
\end{document}

_